\documentclass[12pt,a4paper]{article}

\RequirePackage{etex}
\usepackage[utf8]{inputenc}
\usepackage{amsmath, amssymb, amsthm, amsfonts, mathrsfs}
\usepackage{bbm, commath}
\usepackage{multicol}
\usepackage{blindtext}
\usepackage{graphicx}
\usepackage{tikz-network}
\usetikzlibrary {positioning, patterns, calc}
\usetikzlibrary{graphs, graphs.standard, quotes}
\usepackage{authblk}
\usepackage{enumitem}
\usepackage[colorlinks=true,linkcolor=black,citecolor=black]{hyperref}
\usepackage{tabularx}
\usepackage{array}
\usepackage{subcaption}

\usepackage{natbib}
 \usepackage{setspace}
 \let\oldbibitem\bibitem
 \renewcommand{\bibitem}{\setstretch{1.2}\oldbibitem}

\newtheorem{thm}{Theorem}

\newtheorem{lem}{Lemma}
\newtheorem{cor}{Corollary}
\newtheorem{exm}{Example}

\newtheorem{rem}{Remark}

\def \e {{\bf e}}
\def \u {{\bf u}}
\def \v {{\bf v}}
 
\def \x {{\bf x}}

\def \e {{\mathbf e}}

\newcommand{\ob}[1]{\left(#1\right)}
\newcommand{\cb}[1]{\left\lbrace #1\right\rbrace}
\newcommand{\tb}[1]{\left[#1\right]}

\title{\it Pretty good plus state transfer in cycles}
\author[1]{Sarojini Mohapatra}
\author[2]{Hiranmoy Pal}
\affil[1,2]{National Institute of Technology Rourkela, India-769008 palh@nitrkl.ac.in}
\date{\today}

\begin{document}
	
	\maketitle

	
		\begin{abstract}
        We investigate fractional revival in graphs with respect to the adjacency, Laplacian, and signless Laplacian matrices. We observe that, under certain conditions, fractional revival is preserved under graph complementation. 
       Then we establish a connection between fractional revival in a graph and in its double cover, and obtain a complete characterization of pretty good plus state transfer in cycles and their complements. This leads to characterizations of pretty good vertex state transfer in weighted paths with potential. 
\\\\
   
     \noindent{\it Keywords:}  Spectra of graphs, Circulant graph, Continuous-time quantum walk, Fractional revival, Pretty good plus state transfer.  
\\\\{\it MSC: 15A16, 05C50, 81P45.}
	\end{abstract}

	
	\newpage

\section{Introduction} 
The transfer of quantum states in quantum spin networks is a fundamental problem in quantum information theory. 
\emph{Continuous-time quantum walks}, introduced by Farhi and Gutmann \cite{farhi}, provide a useful framework for analyzing such quantum transport phenomena.
A continuous-time quantum walk on a graph $G$ is governed  by the transition matrix 
\[U_{M(G)}(t)=\exp{(itM(G))}=\sum_{k\geq 0}\frac{(it)^k}{k!}(M(G))^k,\quad \text{where $t\in\mathbb{R}, $}\]
and M(G) denotes the Hamiltonian associated with $G.$ In this context, we consider $M(G)$ to be the adjacency matrix, the Laplacian matrix, or the signless Laplacian matrix of $G.$ Whenever the context is clear, we write 
$U(t)$ instead of 
$U_{M(G)}(t).$ Let $G (V(G), E(G), w)$ be a weighted undirected graph with vertex set $V(G)$ and edge set $E(G),$ where the weight function $w:E(G)\to \mathbb{R}$ assigns a real number to each edge of $G.$ The \textit{adjacency matrix} $A(G)$ is defined by $A(G)_{j,k}=w(j,k)$ for $(j,k)\in E(G),$ and $0$ otherwise. A \textit{potential on $G$} is a diagonal matrix $\Delta,$ where $\Delta_{j,j}$  is the potential at the vertex $j.$ For a graph $G$ with potential, the Hamiltonian $M(G)$ is considered as $A(G)+\Delta.$
Unless otherwise specified, we assume that $G$ is finite, simple, and undirected. In the case of a simple graph, each edge has weight $1,$ and the potential $\Delta$ is the zero matrix. The \textit{Laplacian} and \textit{signless Laplacian} matrices of a simple graph are defined by $L(G)=D(G)-A(G),$ and $Q(G)=D(G)+A(G),$ respectively, where $D(G)$ is the degree matrix of $G.$  Let $\lambda_1,\lambda_2,\ldots,\lambda_d$ be the distinct eigenvalues of $M(G)$ with corresponding eigenprojection matrices $E_{\lambda_1},E_{\lambda_2},\ldots,E_{\lambda_d}.$ Then the spectral decomposition of $U_{M(G)}(t)$ is given by
   \[U_{M(G)}(t)=\sum_{j=1}^{d}\exp{(it\lambda_j)}E_{\lambda_j}.\]

A \textit{real pure state} is represented by a unit vector in $\mathbb{R}^n$ \cite{god7}. Let $a$ and $b$ be two vertices in a graph $G.$ The characteristic vector $\e_a$ is called the \textit{vertex state} associated with $a.$ For a non-zero real number $s,$ a real pure state of the form $\frac{1}{\sqrt{1+s^2}}\ob{\e_a+s\e_b}$ is called an \textit{$s$-pair state} \cite{kim}. In particular, when $s=-1,$ the state is called a \textit{pair state} and for $s=1,$ it is called a \textit{plus state}. 
The \textit{eigenvalue support} of a state $\u$ relative to $M(G)$ is the set 
$\{\lambda_j:E_{\lambda_j}\u\neq 0\}.$
A state $\u$ is called a fixed state if and only if $\u$ is an eigenvector of $M(G)$ associated with the lone eigenvalue in the eigenvalue support of $\u$ \cite{god7}.
A graph $G$ is said to exhibit \emph{perfect state transfer (PST)} between two linearly independent real pure states $\u$ and $\v$ if there exists a
time $\tau$ and a phase factor $\gamma\in\mathbb{C}$ with $|\gamma|=1$ such that
\begin{equation}\label{5eq1}
  U_{M(G)}(\tau)\u=\gamma\v.  
\end{equation} 
If $\u=\v,$ then the state 
$\u$ is said to be periodic in 
$G.$
When $\u$ and $\v$ in \eqref{5eq1} are both vertex states, pair states, or plus states, then the corresponding PST is called \textit{vertex PST, pair PST}, or \textit{plus PST}, respectively.  PST in quantum spin networks was first introduced by Bose \cite{bose} and has attracted considerable attention over the past two decades, with extensive research devoted to characterizing vertex PST.
A fundamental result due to Godsil \cite{god2} establishes that for any integer $k>0,$ there is a finite number of connected graphs of maximum degree $k$ with vertex PST. This result has prompted the investigation of more general notions of state transfer beyond vertex states, 
such as pair (plus) PST \cite{che1,  jia, ojha1, pal10}, $s$-pair PST \cite{kim}, PST between real pure states \cite{god8, god7,pal11}.  

A relaxation to vertex PST known as \emph{pretty good state transfer (PGST)} was introduced in \cite{god1, vin}. Extension of this notion to more general states, including pair states and real pure states, have also been investigated \cite{god8, pal11, ojha1, pal10}. A graph  $G$ is said to exhibit \emph{PGST} between two linearly independent real pure states $\u$ and $\v$ if there exists a sequence of real number $t_k\in\mathbb{R},$ and a complex number $\gamma$ of unit modulus such that 
\begin{equation}\label{5eq2}
       \lim_{k \to \infty}U_{M(G)}\ob{t_k}\u=\gamma \v.
   \end{equation}
   As in vertex PST, \cite{god7} shows that plus PST with respect to the adjacency, Laplacian, and signless Laplacian matrices is a rare phenomenon. This motivates us to investigate plus PGST in graphs. In particular, we study plus PGST in cycles and their complements.


\emph{Fractional revival (FR)} is another generalization of PST, which is relevant for quantum entanglement generation. FR occurs whenever a continuous-time quantum walk maps
the characteristic vector of a vertex to a superposition of the characteristic vectors of a
subset of vertices containing the initial vertex.
 In a path $P_n$ on $n$ vertices, FR occurs between two vertices with respect to the adjacency matrix if and only if $n\in\{2,3,4\},$ while among cycles only $C_4$ and $C_6$ admit FR \cite{chan2}. Furthermore, no tree admits Laplacian FR except for the paths on two
and three vertices \cite{chan3}.
In this work, we consider the initial state to be a real pure state. Let 
$\u$ and $\v$ be two linearly independent states in a graph $G.$ The graph $G$ is said to exhibit \emph{FR} from $\u$ to $\v$
 at time $\tau$ if there exist complex scalars $\alpha,\beta$ with $\beta\neq 0$
such that
\begin{equation}\label{5e1}
U_{M(G)}(\tau)\u=\alpha\u+\beta\v.
\end{equation}
In this case, we also say that $(\alpha,\beta)$-FR occurs from $\u$ to $\v$ at $\tau.$ 
In particular, if $\u$ and $\v$ both represent pair (plus) states, then the graph is said to have \textit{pair (plus) FR}. If $\u$ and $\v$ in \eqref{5e1} are orthogonal states, then $|\alpha|^2+|\beta|^2=1.$
If $\alpha=0$ in \eqref{5e1}, then $G$ admits PST between $\u$ and $\v.$  A non-trivial relation between FR from a vertex state and a pair state is established in \cite{pal10}.
A relaxation to FR called as \textit{pretty good fractional revival (PGFR)} between vertex states was introduced by Chan et al. in \cite{cha3}. 
  We consider PGFR from a real pure state $\u$ to $\v,$ where $\u$ and $\v$ are linearly independent.
   A graph $G$ is said to exhibit PGFR from $\u$ to $\v$ if there is a sequence of time $t_k\in\mathbb{R}$ such that for some $\alpha,\beta\in\mathbb{C}$ with $\beta\neq 0,$
   \[
       \lim_{k \to \infty}U_{M(G)}\ob{t_k}\u=\alpha\u+\beta\v.
  \]


\begin{rem}\label{5r1}
    If $G$ is a regular graph, then the transition matrices governed by the adjacency matrix, the Laplacian matrix, or the signless Laplacian matrix differ only by a global phase. Consequently, the state transfer
properties are identical under all such choices of $M(G).$ 
\end{rem}

Throughout the paper, $\mathbf{1}$ denotes the all-ones vector, $I$ the identity matrix, $J$ the all-ones square matrix, and $\mathbf{0}$ the zero matrix of appropriate order. 
The article is organized as follows. In Section \ref{5s2}, we include the effect of certain graph automorphisms on the existence of PGST from plus and vertex states. In Section \ref{5sec2}, we study FR from a real pure state in the complement of a graph with respect to the adjacency, Laplacian, and signless Laplacian matrices. Section \ref{5sec3} establishes a relation between FR in a graph and its double cover. In Section \ref{5sec5}, we provide a complete characterization of plus PGST in cycles and their complements. A few observations on vertex PGST in certain weighted paths with potential is presented in Section \ref{5sec6}. 

\section{Algebraic properties}\label{5s2}
An \emph{automorphism} $f$ of a graph $G$ is a bijection on the vertex set $V(G)$ such that vertices 
$a$ and $b$ are adjacent in $G$ if and only if  $f(a)$ and $f(b)$ are adjacent in $G.$
 If $P$ is the permutation matrix corresponding to the automorphism $f,$ then $P$ commutes with $M(G).$  Since the transition matrix $U_{M(G)}(t)$ is a polynomoal in $M(G),$ the matrix $P$ commutes with $U_{M(G)}(t)$ as well. The automorphism $f$ is said to \emph{fix} a vertex $a$ in $G$ if $P\e_a=\e_a.$ 
 In the following, we observe that the existence of certain automorphisms in a graph $G$ with plus PGST guarantees the existence of pair PGST in $G.$
\begin{thm}\label{5l7}
 Let a graph $G$ admit pretty good plus state transfer between $\frac{1}{\sqrt{2}}\ob{\e_a+\e_b}$ and $\frac{1}{\sqrt{2}}\ob{\e_c+\e_d}.$ Then the graph $G$ admits pretty good pair state transfer if there exists an automorphism of $G$ with  permutation matrix 
 $P$ satisfying one of the following conditions.
 \begin{enumerate}
    \item $P\e_a=\e_a$ and $P\e_b\neq \e_b,$ 
    \item  $P\e_a=\e_b$ and $P\e_b\neq \e_a.$ 
\end{enumerate}
 \end{thm}
 \begin{proof}
  Suppose the graph $G$ admits plus PGST between $\frac{1}{\sqrt{2}}\ob{\e_a+\e_b}$ and $\frac{1}{\sqrt{2}}\ob{\e_c+\e_d}.$ Then by \eqref{5eq2}, we have 
  \begin{equation}\label{5eqn7}
         \lim_{k \to \infty}U_{M(G)}\ob{t_k}\ob{\e_a+\e_b}=\gamma\ob{\e_c+\e_d}.
     \end{equation}
    Suppose condition-1 holds. Then
     multiplying $P$ on both sides of \eqref{5eqn7} and subtracting the resulting equation from \eqref{5eqn7}, yields
 \begin{equation}\label{5eqn8}
        \lim_{k \to \infty}U_{M(G)}\ob{t_k}\ob{\e_b-P\e_b}=\gamma\ob{\e_c+\e_d-P\e_c-P\e_d}.
   \end{equation}
    Since $U_{M(G)}(t)$ is unitary, \eqref{5eqn8} implies that $\left\|\frac{1}{\sqrt{2}}\ob{\e_b-P\e_b}\right\|=\left\|\frac{1}{\sqrt{2}}\ob{\e_c+\e_d-P\e_c-P\e_d}\right\|.$ This equality is possible only if $\frac{1}{\sqrt{2}}\ob{\e_c+\e_d-P\e_c-P\e_d}$ represents a pair state. Hence, $G$ admits pair PGST.  
    Similarly, if condition-2 holds, an analogous argument shows that  
    $G$ admits pair PGST.
 \end{proof}
 Next, we observe that certain graph automorphisms prevent the existence of PGST between a vertex state and a plus state.
 \begin{thm}\label{5th1}
 Let $a, b$ and $c$ be vertices of a graph $G.$ If there exists an automorphism of $G$ that fixes $b$ but not $a,$ then there is no pretty good state transfer between 
$\e_a$ and $\frac{1}{\sqrt{2}}\ob{\e_b+\e_c}.$ 
\end{thm} 
\begin{proof}
    Suppose $G$ admits PGST between $\e_a$ and $\frac{1}{\sqrt{2}}\ob{\e_b+\e_c}.$  Then by \eqref{5eq2}, we have
    \begin{equation}\label{5e29}
        \lim_{k \to \infty}U_{M(G)}\ob{t_k}\e_a=\frac{\gamma}{\sqrt{2}}\ob{\e_b+\e_c}.
    \end{equation}
     Let $P$ be the permutation matrix of the automorphism of $G$ that fixes $b$ but not $a.$  Multiplying $P$ on both sides of \eqref{5e29} and subtracting the resulting equation from \eqref{5e29}, gives
 \[
        \lim_{k \to \infty}U_{M(G)}\ob{t_k}(\e_a-P\e_a)=\frac{\gamma}{\sqrt{2}}\ob{\e_c-P\e_c},
    \]
 which is a contradiction as $\left\|\e_a-P\e_a\right\|\neq \left\|\frac{1}{\sqrt{2}}\ob{\e_c-P\e_c}\right\|.$ 
Hence, the result follows.
\end{proof}

\section{State transfer in graph complement}\label{5sec2}
The \emph{complement} of a graph $G,$ denoted by $\overline{G},$ is the graph on the vertex set of $G,$ where two distinct vertices are adjacent in $\overline {G}$ if and only if they are not adjacent in $G.$ 
The matrix $M(\overline{G})$ corresponding to the adjacency, Laplacian, and signless Laplacian matrices of the complement $\overline{G}$ on $n$ vertices is given by
\[M(\overline{G})=\delta J+\zeta I-M(G),\]
   where $\delta=\left\{ \begin{array}{rcl}
 -1, & \mbox{if} & M =L \\ 
 1, & \mbox{if} & M \in\{A,Q\},
 \end{array}\right.$
\quad and \quad
$\zeta=\left\{ \begin{array}{rcl}
 -1, & \mbox{if} & M=A, \\ 
n, & \mbox{if} & M=L, \\
n-2, & \mbox{if} & M=Q. 
 \end{array}\right.$
\\The preservation of vertex PST under graph complementation was established in \cite[Lemma 15.3]{god1} and \cite[Theorem 2]{alv} for the adjacency matrix of regular graphs and for the Laplacian matrix, respectively. Moreover, the preservation of Laplacian FR between two vertices under graph complementation was shown in \cite[Theorem 9]{chan3}. The observations extend to FR from real pure states as well. We include the proof for convenience.
 \begin{thm}\label{5t15}
 Let $\mathbf{1}$ be an eigenvector of a graph $G$ on $n$ vertices. If $G$ exhibits fractional revival from a state $\u$ to $\v$ at time $\tau$ with respect to $M(G)$ and $n\tau\in 2\pi\mathbb{Z},$ then the complement $\overline G$ exhibits fractional revival from $\u$ to $\v$ relative to $M(\overline{G}).$ 
 \end{thm}
\begin{proof}
   Since $\mathbf{1}$ is an eigenvector of $G$ and $J=\mathbf{1}\mathbf{1}^T,$ it follows that $M(G)$ commutes with $J.$ Therefore, for all $t\in\mathbb{R},$
   \[U_{M(\overline{G})}(t)=\exp{(itM(\overline{G}))=\exp{(i t\zeta)}\exp{(it\delta J)}}\exp{(-itM(G))}.\]
    The spectral decomposition of $J$ gives
   $\exp{(it\delta J)}=[\exp{(itn\delta)}-1]\frac{1}{n}J+I.$ 
   If $n\tau\in 2\pi \mathbb{Z},$ then we obtain $U_{M(\overline{G})}(-\tau)\u=\exp{(-i\tau\zeta)} U_{M(G)}(\tau)\u,$ and the result follows.
\end{proof}
 It is well known that the path $P_2$ exhibits PST between the end vertices at $\frac{\pi}{2}.$ Using Theorem \ref{5t15}, one may observe that the complete graph on $4n$ vertices with a single missing edge admits Laplacian vertex PST between the non-adjacent vertices. Furthermore, the result in \cite[Corollary 1]{pal8} can be recovered as a direct consequence of Theorem \ref{5t15}.

The \emph{join} of two graphs $G$ and $H$ is defined by $G+H:=\overline{\overline{G}\cup\overline{H}},$ where $\overline{G}\cup\overline{H}$ is the disjoint union of $\overline{G}$ and $\overline{H}.$ An analogous observation to \cite[Corollary 10]{chan3} is obtained for real pure states using Theorem \ref{5t15}. 
 \begin{cor}\label{5c3}
 Let the complement of a graph $G$ exhibit fractional revival from $\u$ to $\v$ at time $\tau$ relative to the Laplacian matrix. For any graph $H,$ if $\tau(|V(G)|+|V(H)|)\in 2\pi\mathbb{Z},$
 then the join $G + H$ has fractional revival from $\u$ to $\v$ relative to the Laplacian matrix.
 \end{cor}
 Although the path $P_3$ does not admit Laplacian plus PST \cite[Corollary 7.8]{god7}, we use it to find infinitely many graphs with Laplacian plus PST. 
\begin{exm}\label{5ex1}
Let $P_3$ be a path on three vertices $1,2$ and $3,$ where both $1$ and $3$ are adjacent to $2.$ Since, the path $P_2$ admits Laplacian PST between the end vertices at $\frac{\pi}{2},$ the complement $\overline{P_3}$ admits PST between $\frac{1}{\sqrt{2}}\ob{\e_1+\e_2}$ and $\frac{1}{\sqrt{2}}\ob{\e_3+\e_2}$  at $\frac{\pi}{2}.$ Therefore, by Corollary \ref{5c3},
the graph $P_3+H$ admits Laplacian plus PST at $\frac{\pi}{2}$ between $\frac{1}{\sqrt{2}}\ob{\e_1+\e_2}$ and $\frac{1}{\sqrt{2}}\ob{\e_3+\e_2},$ whenever $|V(H)|\equiv 1 \pmod 4.$ 
\end{exm}
    


\begin{rem}\label{5remark2}
 The observation in Theorem \ref{5t15} can be extended to PGFR. More precisely, if a graph $G$ on $n$ vertices exhibits PGFR from a state $\u$ to $\v$ with respect to the sequence $t_k$ relative to the Laplacian matrix $L(G),$ and $nt_k\in 2\pi\mathbb{Z},$ then the the complement $\overline G$ also exhibits PGFR from $\u$ to $\v$ relative to $L(\overline{G}).$ 
 \end{rem}
For regular graphs, \cite[Lemma 3]{pal11} shows that state transfer is preserved under graph complementation for any real pure state $\u$ with $\mathbf{1}^T\u=0.$ 
 We provide a sufficient condition that extends this observation to non-regular graphs,
which does not require any additional condition on the time parameter. 
\begin{thm}\label{5th2}
Let $G$ be a graph on $n$ vertices. For a vector $\u\in\mathbb{R}^n,$
  if $\mathbf{1}$ is orthogonal to the set of vectors $\cb{\u,M(G)\u,\ldots, M(G)^{n-1}\u},$ then
   \[U_{M(\overline{G})}(t)\u=\exp{(it\zeta)} U_{M(G)}(-t)\u,\]
 where $U_{M(G)}(t)$ and $U_{M(\overline{G})}(t)$ are the transition matrices of $G$ and $\overline{G}$, respectively, and  the constant $\zeta$ is as defined earlier for $M\in\{A,L,Q\}.$
 \end{thm}
 \begin{proof}
 Since $G$ is a graph of order $n,$ the matrix $M(G)^k$ can be expressed as a linear combination of $I, M(G),\ldots, M(G)^{n-1},$ for all non-negative integers $k.$ As $\mathbf{1}$ is orthogonal to the subspace generated by $\cb{\u,M(G)\u,\ldots, M(G)^{n-1}\u},$ it follows that $JM(G)^k\u=\mathbf{0}.$ Now the result follows from $M(\overline{G})^k\u=(\delta J+\zeta I-M(G))^k\u=(\zeta I-M(G))^k\u.$ 
 \end{proof}
 \begin{exm}\label{5ex2}
     The path $P_5$ exhibits pair PST between $\u=\frac{1}{\sqrt{2}}(\e_1-\e_5)$ and $\v=\frac{1}{\sqrt{2}}(\e_2-\e_4)$ at $\frac{\pi}{2}$ with respect to the adjacency matrix \cite{pal10}. Note that $A(P_5)\u=\v$ and  $A(P_5)^2\u=\u.$ Then by Theorem \ref{5th2}, the graph $\overline{P_5}$ admits PST between $\frac{1}{\sqrt{2}}(\e_1-\e_5)$ and $\frac{1}{\sqrt{2}}(\e_2-\e_4).$
 \end{exm}
As a consequence of Theorem \ref{5th2}, we observe that FR is preserved under graph join.
 \begin{cor}\label{5co2}
     Suppose the premise of Theorem \ref{5th2} holds. If the graph $G$ exhibits fractional revival from $\u$ to $\v$ at $\tau,$ with respect to $M(G),$ then for any graph $H,$ the join $G+H$ admits fractional revival from $[\u^T\quad\mathbf{0}]^T$ to $[\v^T\quad\mathbf{0}]^T$ relative to $M(G+H).$   
 \end{cor}
\begin{proof}
It follows from the proof of Theorem \ref{5th2} that the premise of Theorem \ref{5th2} holds for $\overline{G}.$  For any graph $H$ of order $m,$ observe that the same conditions also hold for the disjoint union $\overline{G}\cup \overline{H}$ with the vector $[\u^T\quad\mathbf{0}]^T\in \mathbb{R}^{m+n}.$   
    Since $G$ admits FR from $\u$ to $\v,$ Theorem \ref{5th2} applies to conclude that $\overline{G}\cup \overline{H}$ admits FR from $[\u^T\quad\mathbf{0}]^T$ to $[\v^T\quad\mathbf{0}]^T.$ Consequently, $G+H$ admits FR from $[\u^T\quad\mathbf{0}]^T$ to $[\v^T\quad\mathbf{0}]^T.$
\end{proof}

\begin{exm} In continuation of  Example \ref{5ex2}, one may apply Corollary \ref{5co2} to show that for any graph $H,$ both $P_5+H$ and $\overline{P_5}+H$ admit PST between $\frac{1}{\sqrt{2}}(\e_1-\e_5)$ and $\frac{1}{\sqrt{2}}(\e_2-\e_4).$  
\end{exm}

\section{Double covers}\label{5sec3}
Let $X_1$ and $X_2$ be graphs on the same vertex set. The graph $X_1\ltimes X_2$ is defined by the adjacency matrix \[\begin{bmatrix}
    A(X_1) & A(X_2)\\
    A(X_2) & A(X_1)
\end{bmatrix}.\] 
If $X_1$ and $X_2$ have disjoint edge sets, then $X_1\ltimes X_2$ is called a double cover of the graph with adjacency matrix $A(X_1)+A(X_2).$ 
Let $D(X_1)$ and $D(X_2)$ be the degree matrices of graphs $X_1$ and $X_2,$ respectively. Consider $D=D(X_1)+D(X_2).$ Then the matrix $M\in\{A,L,Q\}$ associated with the graph $X_1\ltimes X_2$ becomes
\[M(X_1\ltimes X_2)=\begin{bmatrix}
    \eta D+\delta A(X_1) & \delta A(X_2)\\
   \delta A(X_2) & \eta D+\delta A(X_1)
\end{bmatrix},\quad \text{where}\]
 \[\eta=\left\{ \begin{array}{rcl}
 0, & \mbox{if} & M=A, \\ 
 1, & \mbox{if} & M\in\{L,Q\},
 \end{array}\right.\quad 
\text{and} \quad \delta=\left\{ \begin{array}{rcl}
 -1, & \mbox{if} & M =L, \\ 
 1, & \mbox{if} & M \in\{A,Q\}. 
 \end{array}\right.\] 

 We extend \cite[Lemma 5.1]{cou} to obtain a unified expression for the transition matrix of $X_1\ltimes X_2$ with respect to the adjacency, Laplacian, and signless Laplacian matrices. 
 \begin{thm}\label{5t19}
    Let $X_1$ and $X_2$ be graphs on the same vertex set $V.$ 
    The transition matrix of  $X_1\ltimes X_2$ is given by 
 \[U_{M({X_1\ltimes X_2})}(t)=\frac{1}{2}\begin{bmatrix}
    U_{M(G_+)}(t)+ U_{M(G_-)}(t) & U_{M(G_+)}(t)- U_{M(G_-)}(t)\\
    U_{M(G_+)}(t)- U_{M(G_-)}(t)& U_{M(G_+)}(t)+ U_{M(G_-)}(t)
\end{bmatrix},\]
where $t\in \mathbb{R},$ and $M(G_{\pm})=\eta D+\delta(A(X_1)\pm A(X_2)).$
\end{thm}
\begin{proof}
  Let $A(K_2)$ be the adjacency matrix of a complete graph on two vertices, then the matrix $M\in\{A,L,Q\}$ associated with the graph $X_1\ltimes X_2$ can be written as
\[M(X_1\ltimes X_2)=I_2\otimes(\eta D+\delta A(X_1))+A(K_2)\otimes  (\delta A(X_2)). \]  
Note that $I_2$ and $A(K_2)$ commute and can be simultaneously diagonalized by the matrix \[H=\frac{1}{\sqrt{2}}\begin{bmatrix}
    1 & 1\\
    1 & -1
\end{bmatrix}.\] 
Since $(H\otimes I_n)^{-1}=H\otimes I_n,$ it follows from \cite[Lemma 3.1]{cou} that
\[(H\otimes I_n)U_{M(X_1\ltimes X_2)}(t)(H\otimes I_n)=\begin{bmatrix}
    U_{M(G_+)}(t) & \mathbf{0}\\
     \mathbf{0} & U_{M(G_-)}(t)
    \end{bmatrix},\]
    where $M(G_{\pm})=\eta D+\delta(A(X_1)\pm A(X_2)),$ and the result follows. 
\end{proof}
A characterization of vertex PST in the graph $X_1\ltimes X_2$
 was established in \cite[Theorem 5.2]{cou}. Furthermore, Laplacian pair PST was studied in \cite[Theorem 4.2]{jia} when both $X_1$ and $X_2$ are regular. We extend these results to the existence of FR in $X_1\ltimes X_2$ with respect to $M\in\{A,L,Q\}.$ 
\begin{thm}\label{5t21}
    Let $X_1$ and $X_2$ be graphs on the same vertex set $V.$ Let $M(G_{\pm})=\eta D+\delta (A(X_1)\pm A(X_2))$ be the corresponding matrices of $G_+$ and $G_-,$ respectively. Then for the states $\u, \v\in \mathbb{R}^{|V|},$ and $\alpha,\beta\in\mathbb{C}$ with $\beta\neq 0,$ the following holds.
   \begin{enumerate}

\item $X_1\ltimes X_2$ admits $(\alpha,\beta)$-fractional revival from $\tb{\u^T \quad \mathbf{0}}^T$
 to $\tb{\mathbf{0} \quad 
\u^T}^T$ at $\tau$ if and only if $\u$ is periodic at $\tau$ with respect to $M(G_+)$ and $M(G_-)$ with phase factor $\alpha+\beta$ and $\alpha-\beta,$ respectively. 

\item $X_1\ltimes X_2$ admits $(\alpha,\beta)$-fractional revival at $\tau$ from state $\tb{\u^T \quad \mathbf{0}}
^T$ to  $\tb{\v^T \quad \mathbf{0}}
^T$ $\ob{\text{or, from} \tb{  \mathbf{0} \quad \u^T}
^T \text{to}   \tb{ 
\mathbf{0} \quad \v^T}^T}$ if and only if  $U_{M(G_\pm)}(\tau)\u=\alpha\u+\beta\v.$ 

\item $X_1\ltimes X_2$ admits $(\alpha,\beta)$-fractional revival at $\tau$ from state $\tb{\u^T \quad \mathbf{0}}
^T$ to $\tb{\mathbf{0} \quad 
\v^T}^T$ $\ob{\text{or, from} \tb{  \mathbf{0} \quad \u^T}
^T \text{to}   \tb{ 
\v^T \quad \mathbf{0}}^T}$ if and only if  $U_{M(G_\pm)}(\tau)\u=\alpha\u\pm \beta\v.$

\item $G_+$ admits $(\alpha,\beta)$-fractional revival at $\tau$ from $\u$ to $\v$ if and only if $X_1\ltimes X_2$ admits $(\alpha,\beta)$-fractional revival at $\tau$ from $\frac{1}{\sqrt{2}}\tb{\u^T \quad \u^T}
^T$ to   $\frac{1}{\sqrt{2}}\tb{\v^T \quad \v^T}
^T.$ 

\item $G_-$ admits $(\alpha,\beta)$-fractional revival at $\tau$ from $\u$ to $\v$ if and only if $X_1\ltimes X_2$ admits $(\alpha,\beta)$-fractional revival at $\tau$ from $\frac{1}{\sqrt{2}}\tb{\u^T \quad -\u^T}
^T$ to   $\frac{1}{\sqrt{2}}\tb{\v^T \quad -\v^T}
^T.$
\end{enumerate} 

\end{thm}



\begin{proof}
  Let $X_1\ltimes X_2$ admit $(\alpha,\beta)$-FR from $\tb{\u^T\quad\mathbf{0}}^T$ to $\tb{\mathbf{0}\quad\u^T}^T$ at $\tau.$ Then  
  \begin{equation}\label{5eq5}
      U_{M(X_1\ltimes X_2)}(\tau) \begin{bmatrix}
    \u \\ \mathbf{0}
\end{bmatrix}=\alpha \begin{bmatrix}
    \u \\ \mathbf{0}
\end{bmatrix}+\beta \begin{bmatrix}
   \mathbf{0} \\ \u
\end{bmatrix}=\begin{bmatrix}
      \alpha\u \\ \beta\u
  \end{bmatrix}.
  \end{equation}

Using Theorem \ref{5t19}, we have 
  \begin{equation}\label{5eqn6}
      U_{M({X_1\ltimes X_2})}(\tau)\cdot \begin{bmatrix}
    \u \\ \mathbf{0}
\end{bmatrix}=\frac{1}{2}\begin{bmatrix}
    \ob{U_{M(G_+)}(\tau)+U_{M(G_-)}(\tau)}\u \\ \ob{U_{M(G_+)}(\tau)-U_{M(G_-)}(\tau)}\u
\end{bmatrix}
.\end{equation}
Comparing \eqref{5eq5} and \eqref{5eqn6}, we obtain $U_{M(G_+)}(\tau)\u=(\alpha+\beta)\u$ and $U_{M(G_-)}(\tau)\u=(\alpha-\beta)\u.$
The converse part of the proof is immediate from \eqref{5eqn6}. This proves the first statement. The remaining statements follow by similar arguments.
\end{proof}
In particular, Theorem \ref{5t21}  relates FR from vertex state to FR from plus and pair states.

\begin{cor}\label{5cor3}
Suppose the premise of Theorem \ref{5t21} holds. If $\{0,1\}\times V$ is the vertex set of $X_1\ltimes X_2,$ then the following holds. 
\begin{enumerate}
\item $G_+$ admits $(\alpha,\beta)$-fractional revival at $\tau$ from vertex $a$ to $b,$ if and only if $X_1\ltimes X_2$ admits $(\alpha,\beta)$-fractional revival at $\tau$ from $\frac{1}{\sqrt{2}}(\e_{(0,a)}+\e_{(1,a)})$ to   $\frac{1}{\sqrt{2}}(\e_{(0,b)}+\e_{(1,b)}).$ 

\item $G_-$ admits $(\alpha,\beta)$-fractional revival at $\tau$ from vertex $a$ to $b$ if and only if $X_1\ltimes X_2$ admits $(\alpha,\beta)$-fractional revival at $\tau$ from $\frac{1}{\sqrt{2}}(\e_{(0,a)}-\e_{(1,a)})$ to $\frac{1}{\sqrt{2}}(\e_{(0,b)}-\e_{(1,b)}).$ 
\end{enumerate}  
\end{cor}

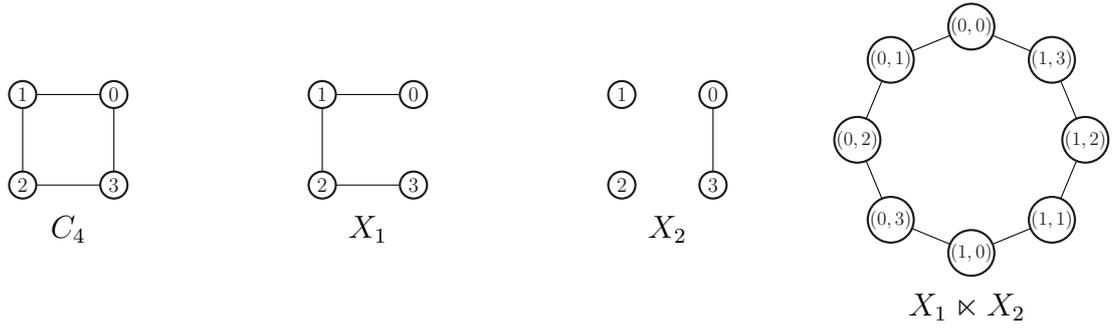
\begin{figure}[]
\centering

\begin{tabular}{cccc}

\begin{minipage}{0.22\textwidth}\centering
\begin{tikzpicture}[scale=1.2,auto=left]
\tikzstyle{every node}=[circle, thick, black!90, fill=white, scale=0.65]
\node[draw,minimum size=0.55cm, inner sep=0 pt] (1) at (1,1) {$0$};
\node[draw,minimum size=0.55cm, inner sep=0 pt] (2) at (0,1) {$1$};
\node[draw,minimum size=0.55cm, inner sep=0 pt] (3) at (0,0) {$2$};
\node[draw,minimum size=0.55cm, inner sep=0 pt] (4) at (1,0) {$3$};
\draw (1)-- (2)-- (3)--(4)-- (1);
\end{tikzpicture}

$C_4$
\end{minipage}
&
\begin{minipage}{0.22\textwidth}\centering
\begin{tikzpicture}[scale=1.2,auto=left]
\tikzstyle{every node}=[circle, thick, black!90, fill=white, scale=0.65]
\node[draw,minimum size=0.55cm, inner sep=0 pt] (1) at (1,1) {$0$};
\node[draw,minimum size=0.55cm, inner sep=0 pt] (2) at (0,1) {$1$};
\node[draw,minimum size=0.55cm, inner sep=0 pt] (3) at (0,0) {$2$};
\node[draw,minimum size=0.55cm, inner sep=0 pt] (4) at (1,0) {$3$};
\draw (1)--(2)--(3)--(4);
\end{tikzpicture}

$X_1$
\end{minipage}
&
\begin{minipage}{0.22\textwidth}\centering
\begin{tikzpicture}[scale=1.2,auto=left]
\tikzstyle{every node}=[circle, thick, black!90, fill=white, scale=0.65]
\node[draw,minimum size=0.55cm, inner sep=0 pt] (1) at (1,1) {$0$};
\node[draw,minimum size=0.55cm, inner sep=0 pt] (2) at (0,1) {$1$};
\node[draw,minimum size=0.55cm, inner sep=0 pt] (3) at (0,0) {$2$};
\node[draw,minimum size=0.55cm, inner sep=0 pt] (4) at (1,0) {$3$};
\draw (4)-- (1);
\end{tikzpicture}

$X_2$
\end{minipage}
&
\begin{minipage}{0.22\textwidth}\centering
\begin{tikzpicture}[ scale=0.5]
\tikzstyle{every node}=[circle, thick, black!90, fill=white, scale=0.6]
\coordinate (top) at (6,2.1);
\coordinate (bottom) at (6,-2.1);
\node[draw,minimum size=0.6cm,inner sep=0 pt] (0) at (3,0) {$(1,2)$};
\node[draw,minimum size=0.6cm,inner sep=0 pt] (1) at (2.12,2.12) {$(1,3)$};
\node[draw,minimum size=0.6cm, inner sep=0 pt] (2) at (0,3){$(0,0)$};
\node[draw,minimum size=0.6cm, inner sep=0 pt] (3) at (-2.12,2.12) {$(0,1)$};
\node[draw,minimum size=0.6cm, inner sep=0 pt] (4) at (-3,0) {$(0,2)$};
\node[draw,minimum size=0.6cm, inner sep=0 pt] (5) at (-2.12,-2.12) {$(0,3)$};
\node[draw,minimum size=0.6cm, inner sep=0 pt] (6) at (0,-3) {$(1,0)$};
\node[draw,minimum size=0.6cm, inner sep=0 pt] (7) at (2.12,-2.12) {$(1,1)$};
\draw (0)--(1)--(2)--(3)--(4)--(5)--(6)--(7)--(0);
\end{tikzpicture}

$X_1 \ltimes X_2$
\end{minipage}

\end{tabular}

\caption{$C_8$ as double cover of $C_4.$}
\label{5fig3}
\end{figure}

One may observe from Corollary~\ref{5cor3} that if a graph $G$ admits vertex PST with respect to $M(G)$, then every double cover $X_1\ltimes X_2$ of $G$ exhibits plus PST with respect to $M(X_1\ltimes X_2).$ Consider the cycle $C_4$, shown in Figure~\ref{5fig3}, together with two spanning subgraphs $X_1$ and $X_2$. It is well known that $C_4$ has vertex PST between the vertices $0$ and $2$ at time $\frac{\pi}{2}$. By Corollary~\ref{5cor3}(1), the double cover $X_1\ltimes X_2$ of $C_4$, which is the cycle $C_8$, exhibits plus PST at time $\frac{\pi}{2}$ between $\frac{1}{\sqrt{2}}\big(\e_{(0,0)}+\e_{(1,0)}\big)$ and $\frac{1}{\sqrt{2}}\big(\e_{(0,2)}+\e_{(1,2)}\big),$  which is consistent with \cite[Theorem 6.5]{kim}.

\section{Plus PGST in cycles}\label{5sec5}


Let $(\Gamma,+)$ be a finite abelian group, and let $S \subseteq\Gamma\backslash\{0\}$ be a connection set satisfying $\left\lbrace -s:s\in S\right\rbrace=S.$ The \emph{Cayley graph}, denoted by $\text{Cay}(\Gamma, S),$ has the vertex set $\Gamma$ where two vertices $a$ and $b$ are adjacent if and only if $a-b\in S$. The Cayley graph over $\mathbb{Z}_n,$ the group of integers modulo $n,$ is called a \emph{circulant graph}. In particular, the cycle $C_n$ is a circulant graph over $\mathbb{Z}_n$ with the connection set $\{1,n-1\}.$
Since a circulant graph is regular,
we investigate plus PGST with respect to the adjacency matrix.
 If $n$ is odd, then there exists an automorphism of $\text{Cay}(\mathbb{Z}_n, S)$ fixing only one vertex. Hence, by Theorem \ref{5l7}, if $\text{Cay}(\mathbb{Z}_n, S)$ admits plus PGST, then the graph must exhibit pair PGST. However, there is no pair PGST in a Cayley graph over an abelian group of odd order \cite [Theorem 10]{pal10}, leading to the following result.   

\begin{lem}\label{5l8}
    A circulant graph of odd order does not exhibit pretty good plus state transfer.
\end{lem}

A pair of vertices $a$ and $b$ in a circulant graph $\text{Cay}(\mathbb{Z}_n, S)$ are called \emph{antipodal vertices} if $a-b=\frac{n}{2}.$ The existence of plus PGST from a pair of non-antipodal vertices in the circulant graph of even order is shown in the following result. 
\begin{lem}\label{5l9}
Let $a$ and $b$ be non-antipodal vertices of the circulant graph $\text{Cay}(\mathbb{Z}_n, S),$ where $n$ is even. If $\text{Cay}(\mathbb{Z}_n, S)$ exhibits pretty good plus state transfer between $\frac{1}{\sqrt{2}}\ob{\e_a+\e_b}$ and $\frac{1}{\sqrt{2}}\ob{\e_c+\e_d},$ then $\e_c+\e_d=\e_{a+\frac{n}{2}}+\e_{b+\frac{n}{2}}.$
\end{lem}
\begin{proof}
     Without loss of generality, let $a=0$ and $b<\frac{n}{2}.$ Suppose $\text{Cay}(\mathbb{Z}_n, S)$ admits plus PGST between $\frac{1}{\sqrt{2}}\ob{\e_0+\e_b}$ and $\frac{1}{\sqrt{2}}\ob{\e_c+\e_d}.$   Consider the automorphism of $\text{Cay}(\mathbb{Z}_n, S)$ with permutation matrix $P$ satisfying $P\e_j=\e_{n-j}$ that fixes only $0$ and $\frac{n}{2}.$ Since $b\neq \frac{n}{2},$ it follows from \eqref{5eqn8} as argued in the proof of Theorem \ref{5l7}, that $\frac{1}{\sqrt{2}}\ob{\e_c+\e_d-P\e_c-P\e_d}$ must be a pair state. If $\e_c=P\e_d,$ then $\e_d=P\e_c,$ which yields  $\frac{1}{\sqrt{2}}\ob{\e_c+\e_d-P\e_c-P\e_d}=0,$ a contradiction.  Consequently, the automorphism must fix either $c$ or $d,$ that is, one of them lies in the set $\{0, \frac{n}{2}\}.$ 
Therefore, without loss of generality, assume that $c\in\{0,\frac{n}{2}\}.$ 
\\\textit{Case I} $(c=0)$:
When $c=0,$ the graph $\text{Cay}(\mathbb{Z}_n, S)$  admits plus PGST between $\frac{1}{\sqrt{2}}\ob{\e_0+\e_b}$ and $\frac{1}{\sqrt{2}}\ob{\e_0+\e_d}.$
Consider an automorphism of $\text{Cay}(\mathbb{Z}_n, S)$ with permutation matrix $P'$ such that $P'\e_j=\e_{2b-j}$ which fixes only $b$ and $ b+\frac{n}{2}.$ Using a similar argument as before, we observe that $\frac{1}{\sqrt{2}}\ob{\e_0+\e_d-P'\e_0-P'\e_d}$ must be a pair state. Since $P'$ does not fix $\e_0,$ it must fix $\e_d.$ 
Consequently, $d=b+\frac{n}{2},$ and plus PGST occurs in $\text{Cay}(\mathbb{Z}_n, S)$  between $\frac{1}{\sqrt{2}}\ob{\e_0+\e_b}$ and $\frac{1}{\sqrt{2}}\ob{\e_0+\e_{b+\frac{n}{2}}}.$ Now, consider another automorphism with permutation matrix $P''$ such that $P''\e_j=\e_{b+j}.$ Since $b<\frac{n}{2},$ we have $ P''\e_b\neq\e_0.$ Using an argument similar to that in the proof of Theorem \ref{5l7}, we obtain 
   \begin{equation}\label{5eq18}
       \lim_{k \to \infty}U\ob{t_k}\ob{\e_0-\e_{2b}}=\gamma\ob{\e_0+\e_{b+\frac{n}{2}}-\e_b-\e_{2b+\frac{n}{2}}}.
   \end{equation} 
   Since $\e_0\notin\cb{ \e_b,\e_{b+\frac{n}{2}}},$ we have  $\e_{b+\frac{n}{2}}\neq \e_{2b+\frac{n}{2}}.$ Then the equality in \eqref{5eq18} holds only if $\e_0=\e_{2b+\frac{n}{2}},$ which implies $b=\frac{n}{4}.$ Consequently, pair PGST occurs in $\text{Cay}(\mathbb{Z}_n, S)$ from $\frac{1}{\sqrt{2}}\ob{\e_0-\e_\frac{n}{2}},$ which is a contradiction to \cite[Lemma 1]{pal10} as there is an automorphism in $\text{Cay}(\mathbb{Z}_n, S)$ that fixes only $0$ and $\frac{n}{2}.$ 
 \\\textit{Case II} $(c=\frac{n}{2}):$ If $c=\frac{n}{2},$ then $\text{Cay}(\mathbb{Z}_n, S)$  admits plus PGST between $\frac{1}{\sqrt{2}}\ob{\e_0+\e_b}$ and $\frac{1}{\sqrt{2}}\ob{\e_\frac{n}{2}+\e_d}.$ Since the automorphism with permutation matrix $P'$ fixes only $b$ and $b+\frac{n}{2},$ it follows that $\frac{1}{\sqrt{2}}\ob{\e_\frac{n}{2}+\e_d-\e_{2b-\frac{n}{2}}-\e_{2b-d}}$ is a pair state as observed from \eqref{5eqn8}. If  $d= 2b-\frac{n}{2},$ then  the state is $\mathbf{0},$ a contradiction. Since $\e_\frac{n}{2}\neq \e_{2b-\frac{n}{2}},$ then the only possibility is either $d=b$ or $d= b+\frac{n}{2}.$ If $d=b,$ then plus PGST occurs between $\frac{1}{\sqrt{2}}\ob{\e_0+\e_b}$ and $\frac{1}{\sqrt{2}}\ob{\e_\frac{n}{2}+\e_b}.$ Using the permutation matrix $P''$ and applying a similar argument as before, we find pair PGST occurs from $\frac{1}{\sqrt{2}}\ob{\e_\frac{n}{4}-\e_\frac{3n}{4}},$ that is from a pair of antipodal vertices. This is again a contradiction to \cite[Lemma 1]{pal10}. Therefore, we must have $d=b+\frac{n}{2}.$
  \end{proof} 

We now characterize plus PGST for the case when 
$a$ and 
$b$ are antipodal vertices in a circulant graph of even order.
\begin{lem}\label{5l10}
    Let $n$ be even and $a$ and $b$ be antipodal vertices in the circulant graph $\text{Cay}(\mathbb{Z}_n, S).$ If there is pretty good plus state transfer from $\frac{1}{\sqrt{2}}\ob{\e_a+\e_b},$ then $n$ is divisible by $4$ and pretty good plus state transfer occurs between $\frac{1}{\sqrt{2}}\ob{\e_a+\e_b}$ and $\frac{1}{\sqrt{2}}\ob{\e_{a+\frac{n}{4}}+\e_{b+\frac{n}{4}}}.$  
\end{lem}

\begin{proof}
    Without loss of generality, let $a=0$ and $b=\frac{n}{2}.$ Suppose $\text{Cay}(\mathbb{Z}_n, S)$ admits plus PGST between $\frac{1}{\sqrt{2}}\ob{\e_0+\e_{\frac{n}{2}}}$ and $\frac{1}{\sqrt{2}}\ob{\e_c+\e_d}.$Then by \eqref{5eq2}, we have 
\begin{equation}\label{5e13}
       \lim_{k \to \infty}U\ob{t_k}\ob{\e_0+\e_{\frac{n}{2}}}=\gamma\ob{\e_c+\e_d}.
   \end{equation}
    Let $P$ be the permutation matrix corresponding to the automorphism of $\text{Cay}(\mathbb{Z}_n, S)$ satisfying $P\e_j=\e_{j+\frac{n}{2}}.$ Since $P(\e_0+\e_{\frac{n}{2}})=\e_0+\e_{\frac{n}{2}},$ and the limit in \eqref{5e13} is unique, we have
$P(\e_c+\e_d)=\e_c+\e_d.$
Therefore, $c$ and $d$ are antipodal vertices. Let $P'$ be the permutation matrix corresponding to the automorphism of $\text{Cay}(\mathbb{Z}_n, S)$ such that $P'\e_j=\e_{n-j} $ that fixes only $0$ and $\frac{n}{2}.$ Multiplying $P'$ on both sides of \eqref{5e13}, gives 
\begin{equation}\label{5e15}
P'(\e_c+\e_d)=\e_c+\e_d.
\end{equation}
The only antipodal vertices satisfying $\eqref{5e15},$ are $c=\frac{n}{4}$ and $d=\frac{3n}{4}.$ 
\end{proof}
 The eigenvalues and eigenvectors of a cycle are well known. Suppose $\omega_n=\exp {\ob{\frac{2\pi i}{n}}}$
is
the primitive $n$-th root of unity. Then the eigenvalues of $C_n$ \cite{bro} are
$\lambda_l=2\cos\ob{{\frac{2l\pi}{n}}},$ where $ 0\leq l\leq n-1,   
$
and the corresponding eigenvectors are \begin{equation}\label{5eq7}\x_l =\tb{1\quad \omega_n^l\quad\cdots\quad\omega_n^{l(n-1)}}^T.
\end{equation}
The complement $\overline{C_n}$ has eigenvalues $\mu_0=n-\lambda_0-1,$ and $\mu_l=-1-\lambda_l,$ for $1\leq l \leq n-1,$ corresponding to the eigenvectors as given in \eqref{5eq7}.
A complete characterization of vertex PGST and pair PGST in $C_n$ is given in \cite[Theorem 13]{pal4} and \cite[Theorem 10]{pal10} as follows.

\begin{thm}\cite{pal4}\label{5t24} A cycle $C_n$ as well as its complement $\overline{C_n}$ admit pretty good state transfer
if and only if $n=2^k$ for $k\geq 2,$ and it occurs between every pair of antipodal vertices.
\end{thm}

\begin{thm}\cite{pal10}\label{5th25}
 A cycle $C_n$ on $n$ vertices admits pretty good pair state transfer if and only
if either $n = 2^k$ or $n = 2^kp,$ where $k$ is a positive integer and $p$ is an odd prime.
\end{thm}
We obtain the following as an immediate consequence of Theorem \ref{5t24}.
\begin{lem}\label{5t26}
Let $C_n$ be a cycle with two vertices $a$ and
$b,$ where $n=2^k,$ with $k\geq 2.$ Pretty good plus state transfer occurs in $C_n$ and its complement $\overline{C_n}$ between $\frac{1}{\sqrt{2}}\ob{\e_a+\e_b}$ and $\frac{1}{\sqrt{2}}\ob{\e_{a+\frac{n}{2}}+\e_{b+\frac{n}{2}}}$ if and only
if $a-b\neq \frac{n}{2}.$
\end{lem}
An even cycle $C_n$ can be realized as a double cover of $C_{\frac{n}{2}}$ as demonstrated in Figure \ref{5fig3}. Let $a$ and $b$ be antipodal vertices in $C_n.$ Corollary \ref{5cor3}(1) shows that a cycle $C_n$ admits PGST between $a$ and $b$ if and only if $C_{2n}$ admits PGST between $\frac{1}{\sqrt{2}}\ob{\e_{(0,a)}+\e_{(1,a)}}$ and $\frac{1}{\sqrt{2}}\ob{\e_{(0,b)}+\e_{(1,b)}}.$ Here, the vertices $(0,a)$ and $(1,a)$ are antipodal vertices in $C_{2n}.$ Including the cycle $C_4$ that admits plus PST between $\frac{1}{\sqrt{2}}\ob{\e_0+\e_2}$ and $\frac{1}{\sqrt{2}}\ob{\e_1+\e_3}$ at $\frac{\pi}{4}$ as shown in \cite[Theorem 6.5]{kim}, and using Lemma \ref{5l10} and Theorem \ref{5t24}, we observe the following.

\begin{lem}\label{5lemma5}
    Let $a$ and $b$ be antipodal vertices in $C_n.$ The cycle $C_n$ admits PGST between $\frac{1}{\sqrt{2}}\ob{\e_a+\e_b}$ and $\frac{1}{\sqrt{2}}\ob{\e_{a+\frac{n}{4}}+\e_{b+\frac{n}{4}}}$ if and only if $n=2^k$ with $k\geq 2.$ 
\end{lem}

Now we investigate plus PGST in $C_n,$ whenever $n$ has an odd prime factor.  
\begin{lem}\label{5l11}
   A cycle $C_n$ does not admit pretty good plus state transfer 
   whenever $n=mp,$ where $m$ is a positive integer and $p$ is an odd prime.  
\end{lem}
\begin{proof}
If $m$ is odd, then the result follows from Lemma \ref{5l8}. Suppose that $m$ is even. Then Lemma \ref{5t26}
 and Lemma \ref{5lemma5} implies that if $C_n$ admits plus PGST from $\frac{1}{\sqrt{2}}\ob{\e_a+\e_b},$ then the vertices $a$ and $b$ must be non-antipodal, and PGST occurs between $\frac{1}{\sqrt{2}}\ob{\e_a+\e_b}$ and $\frac{1}{\sqrt{2}}\ob{\e_{a+\frac{n}{2}}+\e_{b+\frac{n}{2}}}.$ 
 Without loss of generaity let $a=0.$ The spectral decomposition of the adjacency matrix of $C_n$ gives $A(C_n)=\sum_{l=0}^{n-1}\lambda_l E_{\lambda_l}$  where $\lambda_l=2\cos{(\frac{2l\pi}{n})},$ $E_{\lambda_l}=\frac{1}{n}\x_l\x_l^*,$ and $\x_l(j)=\omega_n^{lj}$ for $0\leq j\leq n-1$ as given in \eqref{5eq7}. Consider the automorphism of $C_n$ with permutation matrix $P$ satisfying $P\e_j=\e_{j+\frac{n}{2}}.$
  Therefore, we have
  \[(P\x_l)(j)=\x_l\ob{j+\frac{n}{2}}=\omega_n^{l\ob{j+\frac{n}{2}}}=(-1)^l\x_l(j).\]
  This shows that $P\x_l=(-1)^l \x_l,$ and consequently, $PE_{\lambda_l}=(-1)^lE_{\lambda_l}.$
    Since $P^2=I,$ it follows that $P^T=P$ and $PE_{\lambda_l}=E_{\lambda_l}P.$ Hence,
    \begin{equation}\label{5eq22}
        E_{\lambda_l}\ob{\e_0+\e_b}=PE_{\lambda_l}\ob{e_{\frac{n}{2}}+\e_{b+\frac{n}{2}}}=(-1)^lE_{\lambda_l}\ob{e_{\frac{n}{2}}+\e_{b+\frac{n}{2}}}.
         \end{equation}
In the proof of \cite[Lemma 11]{pal4}, we observe that for an odd prime $p,$ the eigenvalues of $C_n$ satisfy the following identity.
         
  \begin{equation}\label{5e25}
    (\lambda_2-\lambda_1)+\sum_{r=1}^\frac{p-1}{2}(\lambda_{mr+2}-\lambda_{mr+1})+ \sum_{r=1}^\frac{p-1}{2}(\lambda_{mr-2}-\lambda_{mr-1})=0.
 \end{equation}
 Suppose $C_n$ admits plus PGST between $\frac{1}{\sqrt{2}}\ob{\e_0+\e_b}$ and $\frac{1}{\sqrt{2}}\ob{\e_{\frac{n}{2}}+\e_{b+\frac{n}{2}}},$ then there exists a sequence $t_k\in\mathbb{R}$ and a complex number $\gamma$ of unit modulus such that 
\begin{equation}\label{5eq27}
      \gamma \sum_{l=0}^{n-1}E_{\lambda_l}\ob{\e_{\frac{n}{2}}+\e_{b+\frac{n}{2}}}=\gamma\ob{\e_{\frac{n}{2}}+\e_{b+\frac{n}{2}}}=\lim_{k \to \infty}\sum_{l=0}^{n-1}\exp{(it_k\lambda_l)}E_{\lambda_l}\ob{\e_0+\e_b}.
   \end{equation}
 Multiplying both sides of \eqref{5eq27} by $E_{\lambda_l}$ and using \eqref{5eq22}, we obtain  $\displaystyle\lim_{k \to \infty}\exp{(it_k\lambda_l)}=(-1)^l\gamma.$ 
 Therefore, 
$ \displaystyle\lim_{k \to \infty}\exp{(it_k(\lambda_{l+1}-\lambda_l))}=-1. $ Denoting the term on the left-hand side of \eqref{5e25} as $W,$ we obtain  
  $\displaystyle\lim_{k \to \infty}\exp{\ob{it_kW}}=-1.$
  This contradicts \eqref{5e25} as $W=0,$ implying no plus PGST between $\frac{1}{\sqrt{2}}\ob{\e_0+\e_b}$ and $\frac{1}{\sqrt{2}}\ob{\e_{\frac{n}{2}}+\e_{b+\frac{n}{2}}}$ in $C_n$.
\end{proof}
Combining Lemma \ref{5t26}, Lemma \ref{5lemma5} and Lemma \ref{5l11}
together with the case of $C_4$ and $C_8$ that exhibit plus PST, we obtain the following result on plus PGST in cycles.
\begin{thm}\label{5thm9}
    A cycle $C_n$ admits pretty good plus state transfer if and only if $n=2^k$ with $k\geq 2.$ Moreover, every plus state in $C_{2^k}$ exhibits pretty good plus state transfer. 
\end{thm}

Next, we investigate plus PGST in $\overline{C_n}$ whenever $n$ has an odd prime factor.

\begin{lem}\label{5lemm6}
There is no pretty good plus state transfer in the complement $\overline{C_n},$ whenever $n=mp,$ where $m$ is a positive integer and $p$ is an odd prime.
\end{lem}
\begin{proof}
If $m$ is odd, then the result follows from Lemma \ref{5l8}. Thus, assume that $m$ is even. 
The spectral decomposition of the adjacency matrix of
 $\overline{C_n}$ gives $A(\overline{C_n})=\displaystyle\sum_{l=0}^{n-1}\mu_l E_{\mu_l},$ where   
     $E_{\mu_l}=\frac{1}{n}\x_l\x_l^*,$ and $\x_l(j)=\omega_n^{lj},$ for $0\leq j\leq n-1,$ as given in \eqref{5eq7}.
\\\textit{Case I} $(a-b\neq\frac{n}{2})$:   If  $\overline{C_n}$ admits plus PGST from $\frac{1}{\sqrt{2}}\ob{\e_a+\e_b},$ where $a$ and 
$b$ are non-antipodal vertices, then by Lemma \ref{5l9}, it occurs between $\frac{1}{\sqrt{2}}\ob{\e_a+\e_b}$ and $\frac{1}{\sqrt{2}}\ob{\e_{a+\frac{n}{2}}+\e_{b+\frac{n}{2}}}.$ Without loss of generality, let $a=0$ and $0<b<\frac{n}{2}.$
 Let $P$ be the permutation matrix corresponding to the automorphism of $\overline{C_n}$ defined by $P\e_j=\e_{j+\frac{n}{2}}.$ Using a similar approach as in the proof of Lemma \ref{5l11}, we obtain 
\begin{equation}\label{5eq33}
        E_{\mu_l}(e_0+\e_b)=(-1)^l E_{\mu_l}\ob{e_{\frac{n}{2}}+\e_{b+\frac{n}{2}}}.
         \end{equation}
         For an odd prime $p,$ using \eqref{5e25} and the fact that $\lambda_0=n-\mu_0-1,$ and $\lambda_l=-1-\mu_l,$ for $1 \leq l \leq n - 1,$ we have
  \begin{equation}\label{5eqn17}
    (\mu_2-\mu_1)+\sum_{r=1}^\frac{p-1}{2}(\mu_{mr+2}-\mu_{mr+1})+ \sum_{r=1}^\frac{p-1}{2}(\mu_{mr-2}-\mu_{mr-1})=\left\{ \begin{array}{rcl}
 2p, & \mbox{if} & m =2, \\ 
 0, & \mbox{if} & m\geq 3.
 \end{array}\right.
 \end{equation}
 If $\overline{C_n}$ admits plus PGST between $\frac{1}{\sqrt{2}}\ob{\e_0+\e_b}$ and $\frac{1}{\sqrt{2}}\ob{\e_{\frac{n}{2}}+\e_{b+\frac{n}{2}}},$ then a similar argument as in the proof of Lemma \ref{5l11} gives $\displaystyle\lim_{k \to \infty}\exp{(it_k\mu_l)}=(-1)^l\gamma.$ 
 It follows that
$ \displaystyle\lim_{k \to \infty}\exp{(it_k(\mu_{l+1}-\mu_l))}=-1. $ Let $W'$ be the expression on the left-hand side of \eqref{5eqn17}. Then  
  $\displaystyle\lim_{k \to \infty}\exp{\ob{it_kW'}}=-1,$ which contradicts \eqref{5eqn17} for $m\geq 3$ as $W'=0.$

Next we consider the case $m=2.$ Since $\mu_p=1,$ it follows that $\displaystyle\lim_{k \to \infty}\exp{(it_k)}=-\gamma.$ Since $W'=2p,$ we have $\displaystyle\lim_{k \to \infty}\exp(it_k 2p)=-1,$ and consequently $\gamma^{2p}=-1.$ For the eigenvalues $\mu_1$ and $\mu_{p-1},$ we obtain $\displaystyle\lim_{k \to \infty}\exp(it_k(\mu_1+\mu_{p-1}) )=-\gamma^2.$ Since $\mu_1+\mu_{p-1}=-2,$ we have $\gamma^4=-1,$ which is a contradiction to $\gamma^{2p}=-1$ as $p$ is an odd prime.
 \\\textit{Case II} $(a-b=\frac{n}{2})$: If $\overline{C_n}$ admits plus PGST from $\frac{1}{\sqrt{2}}\ob{\e_a+\e_b},$ where $a$ and $b$ are antipodal vertices, then by Lemma \ref{5l10}, $n$ is divisible by $4$ and plus PGST occurs between $\frac{1}{\sqrt{2}}\ob{\e_a+\e_b}$ and $\frac{1}{\sqrt{2}}\ob{\e_{a+\frac{n}{4}}+\e_{b+\frac{n}{4}}}.$ Without loss of generality, let $a=0$ and $b=\frac{n}{2}.$
Let $P'$ be the permutation matrix corresponding to the automorphism of $\overline{C_n}$ defined by $P'\e_j=\e_{j+\frac{n}{4}}.$
   Therefore, we have
   \[(P'\x_l)(j)=\x_l\ob{j+\frac{n}{4}}=\omega_n^{l\ob{j+\frac{n}{4}}}=(i)^l\x_l(j).\]
   This shows that $P'\x_l=(i)^l \x_l,$ and consequently, $P'E_{\mu_l}=(i)^lE_{\mu_l}.$
  Since $P'^T=P'^{-1}$ and $\x_l^*P'^T=(i)^{-l}\x_l^*,$ we have $\x_l^*P'=(i)^l\x_l^*.$ It follows that $P'E_{\mu_l}=E_{\mu_l}P'$ and 
     \begin{equation}\label{5eq29}
         E_{\mu_l}\ob{e_0+\e_{\frac{n}{2}}}=P'E_{\mu_l}\ob{e_{\frac{n}{4}}+\e_{\frac{3n}{4}}}=(i)^lE_{\mu_l}\ob{e_{\frac{n}{4}}+\e_{\frac{3n}{4}}}.
          \end{equation}
          From the proof of \cite[Lemma 11]{pal4}, we observe that for an odd prime $p,$ 
\begin{equation}\label{5e22}
 1+2\sum_{r=1}^\frac{p-1}{2}\cos{\ob{\frac{2r\pi}{p}}}=0.
  \end{equation}
 Multiplying both sides of \eqref{5e22} by $2\cos{\ob{\frac{2\pi}{n}}},$ and $2\cos{\ob{\frac{6\pi}{n}}},$ yields
  \[
     \lambda_1+\sum_{r=1}^\frac{p-1}{2}\lambda_{mr+1}+ \sum_{r=1}^\frac{p-1}{2}\lambda_{mr-1}=0, \quad \text{and} \quad \lambda_3+\sum_{r=1}^\frac{p-1}{2}\lambda_{mr+3}+ \sum_{r=1}^\frac{p-1}{2}\lambda_{mr-3}=0,
  \]
respectively. Consequently, we have
 \begin{equation}\label{5eqn21}
    (\lambda_3-\lambda_1)+\sum_{r=1}^\frac{p-1}{2}(\lambda_{mr+3}-\lambda_{mr+1})+ \sum_{r=1}^\frac{p-1}{2}(\lambda_{mr-3}-\lambda_{mr-1})=0.
 \end{equation}
          Substituting $\lambda_l=-1-\mu_l$ in \eqref{5eqn21}, we obtain
  \begin{equation}\label{5eq26}
    (\mu_3-\mu_1)+\sum_{r=1}^\frac{p-1}{2}(\mu_{mr+3}-\mu_{mr+1})+ \sum_{r=1}^\frac{p-1}{2}(\mu_{mr-3}-\mu_{mr-1})=0.
 \end{equation}
 Suppose $\overline{C_n}$ admits plus PGST between $\frac{1}{\sqrt{2}}\ob{\e_0+\e_\frac{n}{2}}$ and $\frac{1}{\sqrt{2}}\ob{\e_{\frac{n}{4}}+\e_{\frac{3n}{4}}},$ then there exists a sequence $t_k\in\mathbb{R}$ and a complex number $\gamma$ of unit modulus such that 
\begin{equation}\label{5eq31}
      \gamma \sum_{l=0}^{n-1}E_{\mu_l}\ob{\e_{\frac{n}{4}}+\e_{\frac{3n}{4}}}=\gamma\ob{\e_{\frac{n}{4}}+\e_{\frac{3n}{4}}}=\lim_{k \to \infty}\sum_{l=0}^{n-1}\exp{(it_k\mu_l)}E_{\mu_l}\ob{\e_0+\e_{\frac{n}{2}}}.
   \end{equation}
  Multiplying both sides of \eqref{5eq31} by $E_{\mu_l}$ and using \eqref{5eq29}, we obtain 
$\displaystyle\lim_{k \to \infty}\exp{(it_k\mu_l)}=(-i)^l\gamma.$
 Therefore, 
$ \displaystyle\lim_{k \to \infty}\exp{(it_k(\mu_{l+2}-\mu_l))}=-1. $ Denoting the term on the left-hand side of \eqref{5eq26} as $W'',$ we arrive at a contradiction
  $\displaystyle\lim_{k \to \infty}\exp{\ob{it_kW''}}=-1$
as $W''=0.$ Hence, no plus PGST occurs between $\frac{1}{\sqrt{2}}\ob{\e_0+\e_{\frac{n}{2}}}$ and $\frac{1}{\sqrt{2}}\ob{\e_{\frac{n}{4}}+\e_{\frac{3n}{4}}}.$ 
\end{proof}
Now we determine the conditions under which plus PGST occurs from a pair of antipodal vertices in the complement $\overline{C_n}.$
\begin{lem}\label{5lem6}
Let $a$ and $b$ be a pair of antipodal vertices in $\overline{C_n},$ the complement of the cycle $C_n.$
The complement $\overline{C_n}$ admits pretty good plus state transfer from $\frac{1}{\sqrt{2}}\ob{\e_a+\e_b}$ if and only if $n=2^k$ with $k\geq 3.$ 
\end{lem}
\begin{proof}
Without loss of generality let $a=0$ and $b=\frac{n}{2}.$
    If $n=2^k,$ with $k\geq 3,$ then $C_n$ as well as its complement admit vertex PGST with respect to the same sequence $t_l$ in $2\pi\mathbb{Z}$ \cite[Theorem 7]{pal4}. Since $C_n$ can be realized as a double cover of $C_{\frac{n}{2}},$ by Corollary \ref{5cor3}(1), the cycle $C_{2^k}$ admits plus PGST between $\frac{1}{\sqrt{2}}\ob{\e_0+\e_{\frac{n}{2}}}$ and $\frac{1}{\sqrt{2}}\ob{\e_{\frac{n}{4}}+\e_{\frac{3n}{4}}},$ whenever $k\geq 4$ with respect to the same sequence. As argued in Remark \ref{5remark2}, since $2^kt_l\in2\pi\mathbb{Z},$ $\overline{C_{2^k}}$
admits plus PGST between $\frac{1}{\sqrt{2}}\ob{\e_0+\e_{\frac{n}{2}}}$ and $\frac{1}{\sqrt{2}}\ob{\e_{\frac{n}{4}}+\e_{\frac{3n}{4}}},$ whenever $k\geq 4.$ For $k=3,$ the cycle $C_8$ has plus PST between $\frac{1}{\sqrt{2}}\ob{\e_0+\e_4}$ and $\frac{1}{\sqrt{2}}\ob{\e_2+\e_6}$ at $\frac{\pi}{2}.$ Hence, by Theorem \ref{5t15}, plus PST occurs  between those states in $\overline{C_8}$.

Conversely, if $\overline{C_n}$ admits plus PGST from $\frac{1}{\sqrt{2}}\ob{\e_0+\e_\frac{n}{2}},$ then by Lemma \ref{5l10}, it occurs between $\frac{1}{\sqrt{2}}\ob{\e_0+\e_\frac{n}{2}}$ and $\frac{1}{\sqrt{2}}\ob{\e_\frac{n}{4}+\e_\frac{3n}{4}}.$  Using the observation above and Lemma \ref{5lemm6}, we find that the only possible case is when $n=2^k$ with $k\geq 2.$
The graph $\overline{C_4}$ does not admit plus PGST between $\frac{1}{\sqrt{2}}\ob{\e_0+\e_2}$ and $\frac{1}{\sqrt{2}}\ob{\e_1+\e_3}$ as $\frac{1}{\sqrt{2}}\ob{\e_0+\e_2}$ is a fixed state in $\overline{C_4}.$ Hence, the result follows.   
\end{proof}

 Now we provide a complete characterization of plus PGST in the cycle $C_n$ and its complement $\overline{C_n}$ by combining Theorem \ref{5thm9}, Lemma \ref{5lemm6} and Lemma \ref{5lem6}. 
 \begin{thm}\label{5th10}
      A cycle $C_n$ as well as its complement $\overline{C_n}$ admit pretty good plus state transfer if and only if $n=2^k$ with $k\geq 2.$ Moreover, for $k\geq 2,$ every plus state in $C_{2^k}$ and for $k\geq 3,$ every plus state in $\overline{C_{2^k}}$ exhibit pretty good plus state transfer.
 \end{thm}
 \section{PGST in weighted paths with potential}\label{5sec6}
 Let $P_n$ be a path on $n$ vertices with vertex set $\{1,2,\ldots,n\},$ where two vertices $j$ and $k$ are adjacent if and only if $|j-k|=1.$ 
  We study the existence of vertex PGST in certain weighted paths with potential using equitable partitions. A partition $\Pi$ of the vertex set of a graph $G$ with disjoint cells $V_1, V_2, \ldots, V_d$ is said to be \textit{equitable} if each vertex in $V_j$ is adjacent to exactly $c_{jk}$ vertices in $V_k,$ where 
$c_{jk}$ is a constant depending only on the cells $V_j$ and $V_k$. 
The graph with the $d$ cells of $\Pi$ as its vertices having adjacency matrix $A(G/\Pi),$ where $A(G/\Pi)_{j,k}=\sqrt{c_{jk}c_{kj}},$ is called the \textit{symmetrized quotient graph} of $G$ over $\Pi,$ and is denoted by $G/\Pi.$ A relation between vertex PGST in $G/\Pi$ and plus PGST in $G$ is observed in \cite[Remark 5]{god8}, which implies that for $V_1=\{a,b\}$ and $V_2=\{c,d\},$ PGST occurs between $V_1$ and $V_2$ in $G/\Pi$ if and only if PGST occurs between $\frac{1}{\sqrt{2}}\ob{\e_a+\e_b}$ and $\frac{1}{\sqrt{2}}\ob{\e_c+\e_d}$ in $G.$

It is well known that PST occurs only on the paths $P_2$ and $P_3$ with respect to the adjacency matrix \cite{god1}. This raises a question of whether PST could be induced between the endpoints of
a path of arbitrary length by placing a suitable potential on the endpoints. Kempton et al. \cite{kem1} showed that there is no potential that induces PST between the end
points of a path of length at least $4.$ However, for a path $P_n$ of any length, there is some choice of potential such that by
placing the potential on each endpoint of $P_n,$ there is PGST between the endpoints \cite{kem2}. 
  Further, it is observed in \cite{kirk5} that for paths on more than two vertices, the set of weights $w\geq 1$ assigned to the end vertices, for which the resulting path does not admit PGST, forms a dense subset of $[1,\infty).$ 
Here, we present a few observations on the existence of PGST on weighted paths with potentials.

\begin{figure}[]
\centering
\begin{minipage}{0.48\textwidth}
\centering
\begin{tikzpicture}[scale=0.6]
\tikzstyle{every node}=[circle, thick, black!90, fill=white, scale=0.65]
\node[draw,minimum size=0.55cm, inner sep=0 pt] (0) at (0,0) {$0$};
\node[draw,minimum size=0.55cm, inner sep=0 pt] (1) at (2,-2) {$1$};
\node[draw,minimum size=0.55cm, inner sep=0 pt] (2) at (4,-2) {$2$};
\fill (5.1,-2) circle (2.5pt);
\fill (5.5,-2) circle (2.5pt);
\fill (5.9,-2) circle (2.5pt);
\node[draw,minimum size=0.55cm, inner sep=0 pt] (3) at (7,-2) {};
\node[draw,minimum size=0.55cm, inner sep=1.5 pt] (4) at (9,-2) {$n-1$};
\node[draw,minimum size=0.55cm, inner sep=0 pt] (5) at (9,2) {$n$};
\node[draw,minimum size=0.55cm, inner sep=0 pt] (6) at (7,2) {};
\node[draw,minimum size=0.55cm, inner sep=0 pt] (7) at (4,2) {};
\fill (5.1,2) circle (2.5pt);
\fill (5.5,2) circle (2.5pt);
\fill (5.9,2) circle (2.5pt);
\node[draw,minimum size=0.55cm, inner sep=1.5 pt] (8) at (2,2) {$\scriptsize 2n-2$};
\draw (0)--(1)--(2);
\draw (3)--(4)--(5)--(6);
\draw (7)--(8)--(0);
\end{tikzpicture}
\subcaption{}
\end{minipage}
\hfill
\begin{minipage}{0.45\textwidth}
\centering
\begin{tikzpicture}[scale=1.1]
\tikzstyle{every node}=[circle, thick, black!90, fill=white, scale=0.65]
\node[draw,minimum size=0.55cm, inner sep=0 pt] (1) at (0,0) {$1$};
\node[draw,minimum size=0.55cm, inner sep=0 pt] (2) at (1,0) {$2$};
\node[draw,minimum size=0.55cm, inner sep=0 pt] (3) at (2,0) {$3$};
\fill (2.6,0) circle (1.5pt);
\fill (3,0) circle (1.5pt);
\fill (3.4,0) circle (1.5pt);
\node[draw,minimum size=0.55cm, inner sep=0 pt] (4) at (4,0) {};
\node[draw,minimum size=0.55cm, inner sep=0 pt] (5) at (5,0) {$n$};
\draw[thick]
(5) .. controls +(60:1.5cm) and +(120:1.5cm) ..
node[pos=0.3, above right] {$1$}
(5);
\draw (1) -- (2) node[midway, above, yshift=0.5pt] {$\sqrt{2}$};
\draw (2) -- (3);
\draw (4) -- (5);
\end{tikzpicture}
\subcaption{}
\vspace{1cm} 
\begin{tikzpicture}[scale=1]
\tikzstyle{every node}=[circle, thick, black!90, fill=white, scale=0.65]
\node[draw,minimum size=0.55cm, inner sep=0 pt] (1) at (-1,0.6) {$a$};
\node[draw,minimum size=0.55cm, inner sep=0 pt] (2) at (-1,-0.6) {$b$};
\node[draw,minimum size=0.55cm, inner sep=0 pt] (3) at (0,0) {$1$};
\node[draw,minimum size=0.55cm, inner sep=0 pt] (4) at (1,0) {$2$};
\fill (1.6,0) circle (1.5pt);
\fill (2,0) circle (1.5pt);
\fill (2.4,0) circle (1.5pt);
\node[draw,minimum size=0.55cm, inner sep=0 pt] (5) at (3,0) {};
\node[draw,minimum size=0.55cm, inner sep=1.5 pt] (6) at (4,0) {$n-1$};
\draw[thick]
(6) .. controls +(60:1.5cm) and +(120:1.5cm) ..
node[pos=0.3, above right] {$1$}
(6);
\draw (1)--(3)--(2);
\draw (3)--(4);
\draw (5)--(6);
\end{tikzpicture}
\subcaption{}
\end{minipage}
\caption{(a) Odd cycle $C_{2n-1},$ (b) Symmetrized quotient graph of $C_{2n-1},$ (c) A graph with symmetrized quotient matrix equal to the adjacency matrix of Figure (b).}
\label{5fig4}
\end{figure}
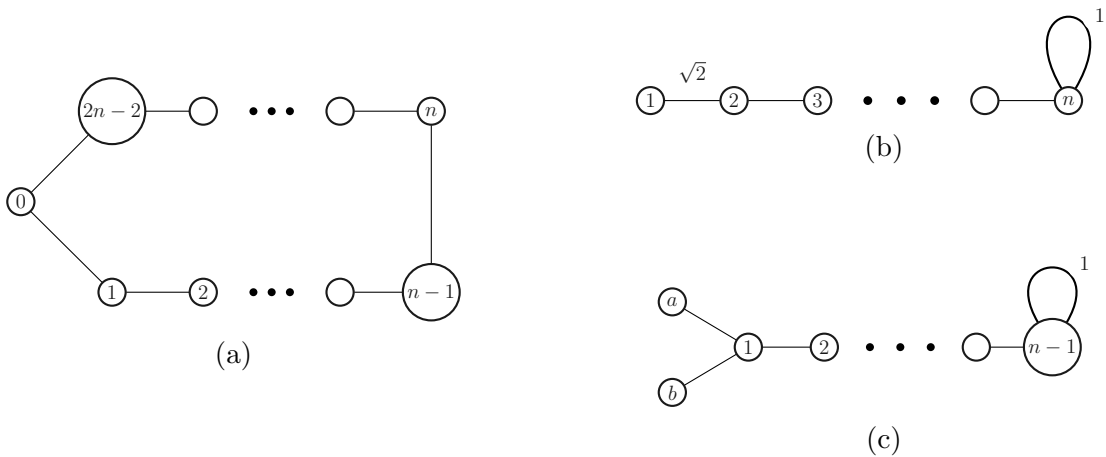
\begin{thm}
    There is no pretty good vertex state transfer in a weighted path on $n$ vertices, with potential $1$ only at the end vertex $n$ and $w(1,2)=\sqrt{2}$ and all other edges having weight $1.$ 
\end{thm}
\begin{proof}
    The odd cycle $C_{2n-1}$ given in Figure \ref{5fig4}(a) has an equitable partition $\Pi$ with cells $V_1=\{0\}, V_2=\{1, 2n-2\},\ldots, V_{n}=\{n-1,n\}.$
  Lemma \ref{5l8} together with \cite[Remark 5]{god8} demonstrates that there is no PGST between any pair of vertices in  $\{2,3,\ldots,n\}$ in the symmetrized quotient graph of $C_{2n-1}$ given in Figure \ref{5fig4}(b). Suppose the graph has PGST between $1$ and $j,$ for some $j\in\{2,3,\ldots,n\}.$ Then by \cite[Remark 5]{god8}, the cycle in Figure \ref{5fig4}(a)
 has PGST between $\e_0$ and $\frac{1}{\sqrt{2}}\ob{\e_a+\e_b},$ where $a$ and $b$ are in $V_j.$ This gives a contradiction to Theorem \ref{5th1} as there is an automorphism of $C_n$ which fixes only the vertex $a.$ Hence the result follows.  
\end{proof}


Note that \cite[Corollary 9.2]{god1} also holds for PGST. Therefore, if $G$ admits PGST between $a$ and $b,$ then each automorphism fixing $a$ must fix $b.$ Suppose there is PGST between $a$ and $j,$ for some $j\in\{1,2,\ldots, n-1\},$ in Figure \ref{5fig4}(c). Since there is an automorphism of the graph in Figure \ref{5fig4}(c) that fixes $j$ but not $a,$ there is no PGST between $a$ and $j.$
 The graph in Figure \ref{5fig4}(c) has an equitable partition $\Pi'$ with cells $V_1=\{a,b\}, V_j=\{j-1\},$ for $2\leq j\leq n.$ The symmetrized quotient matrix relative to $\Pi'$ is equal to the adjacency matrix of the graph given in Figure \ref{5fig4}(b). Therefore, using \cite[Remark 5]{god8}, the result follows.
\begin{cor}
Let $G$ be the graph obtained from $P_n$ by joining two pendant vertices to vertex 
$1$ and assigning a potential 
$1$ only at vertex $n.$ Then there is no pretty good vertex state transfer in
$G$ from the vertices of $P_n.$ 
\end{cor}

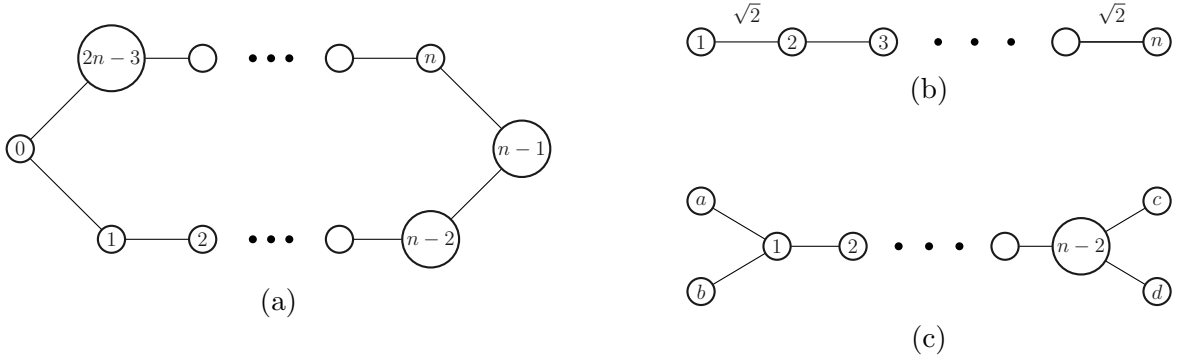
\begin{figure}[]
\centering

\begin{minipage}{0.48\textwidth}
\centering
\begin{tikzpicture}[scale=0.6]
\tikzstyle{every node}=[circle, thick, black!90, fill=white, scale=0.65]

\node[draw,minimum size=0.55cm, inner sep=0 pt] (0) at (0,0) {$0$};
\node[draw,minimum size=0.55cm, inner sep=0 pt] (1) at (2,-2) {$1$};
\node[draw,minimum size=0.55cm, inner sep=0 pt] (2) at (4,-2) {$2$};
\fill (5.1,-2) circle (2.5pt);
\fill (5.5,-2) circle (2.5pt);
\fill (5.9,-2) circle (2.5pt);
\node[draw,minimum size=0.55cm, inner sep=0 pt] (3) at (7,-2) {};
\node[draw,minimum size=0.55cm, inner sep=1.5 pt] (4) at (9,-2) {$n-2$};
\node[draw,minimum size=0.55cm, inner sep=1.5 pt] (5) at (11,0) {$\scriptsize n-1$};
\node[draw,minimum size=0.55cm, inner sep=0 pt] (6) at (9,2) {$\scriptsize n$};
\node[draw,minimum size=0.55cm, inner sep=0 pt] (7) at (7,2) {};
\node[draw,minimum size=0.55cm, inner sep=0 pt] (8) at (4,2) {};
\fill (5.1,2) circle (2.5pt);
\fill (5.5,2) circle (2.5pt);
\fill (5.9,2) circle (2.5pt);
\node[draw,minimum size=0.55cm, inner sep=1.5 pt] (9) at (2,2) {$\scriptsize 2n-3$};

\draw (0)--(1)--(2);
\draw (3)--(4)--(5)--(6)--(7);
\draw (8)--(9)--(0);
\end{tikzpicture}
\subcaption{}
\end{minipage}
\hfill
\begin{minipage}{0.45\textwidth}
\centering

\begin{tikzpicture}[scale=1.2]
\tikzstyle{every node}=[circle, thick, black!90, fill=white, scale=0.65]

\node[draw,minimum size=0.55cm, inner sep=0 pt] (1) at (0,0) {$1$};
\node[draw,minimum size=0.55cm, inner sep=0 pt] (2) at (1,0) {$2$};
\node[draw,minimum size=0.55cm, inner sep=0 pt] (3) at (2,0) {$3$};
\fill (2.6,0) circle (1.2pt);
\fill (3,0) circle (1.2pt);
\fill (3.4,0) circle (1.2pt);
\node[draw,minimum size=0.55cm, inner sep=0 pt] (4) at (4,0) {};
\node[draw,minimum size=0.55cm, inner sep=0 pt] (5) at (5,0) {$n$};
\draw (1) -- (2) node[midway, above, yshift=0.5pt] {$\sqrt{2}$};
\draw (4) -- (5) node[midway, above, yshift=0.5pt] {$\sqrt{2}$};
\draw (2) -- (3);
\draw (4) -- (5);

\end{tikzpicture}
\subcaption{}
\vspace{1cm} 

\begin{tikzpicture}[scale=1]
\tikzstyle{every node}=[circle, thick, black!90, fill=white, scale=0.65]

\node[draw,minimum size=0.55cm, inner sep=0 pt] (1) at (-1,0.6) {$a$};
\node[draw,minimum size=0.55cm, inner sep=0 pt] (2) at (-1,-0.6) {$b$};
\node[draw,minimum size=0.55cm, inner sep=0 pt] (3) at (0,0) {$1$};
\node[draw,minimum size=0.55cm, inner sep=0 pt] (4) at (1,0) {$2$};
\fill (1.6,0) circle (1.5pt);
\fill (2,0) circle (1.5pt);
\fill (2.4,0) circle (1.5pt);
\node[draw,minimum size=0.55cm, inner sep=0 pt] (5) at (3,0) {};
\node[draw,minimum size=0.55cm, inner sep=1.5 pt] (6) at (4,0) {$\scriptsize n-2$};
\node[draw,minimum size=0.55cm, inner sep=0 pt] (7) at (5,0.6) {$c$};
\node[draw,minimum size=0.55cm, inner sep=0 pt] (8) at (5,-0.6) {$d$};
\draw (1)--(3)--(2);
\draw (3)--(4);
\draw (5)--(6)--(7);
\draw (6)--(8);

\end{tikzpicture}
\subcaption{}
\end{minipage}

\caption{(a) Even cycle $C_{2n-2},$ (b) Symmetrized quotient graph of $C_{2n-2},$ (c) A graph
with symmetrized quotient matrix equal to the adjacency matrix of Figure (b).}
\label{5fig5}

\end{figure}

Now we consider the symmetrized quotient graph of an even cycle to have the following observations.
 \begin{thm}
    Let $P_n$ be a weighted path  with $w(1,2)=w(n-1,n)=\sqrt{2}$ and the remaining edges have weight $1.$ Then $P_n$ admits pretty good vertex state transfer if and only if $n-1=2^k$ with $k\geq 1.$
\end{thm}
\begin{proof}
    The even cycle $C_{2n-2}$ given in Figure \ref{5fig5}(a) has an equitable partition with the cells $V_1=\{0\},V_2=\{1, 2n-3\},\ldots,V_{n-1}=\{n-2,n\},V_{n}=\{n-1\}.$ Using Theorem \ref{5th10} and \cite[Remark 5]{god8}, we have the weighted path exhibits PGST if and only if $2n-2=2^k$ with $k\geq 2.$ Hence the result follows.
\end{proof}
  The graph in Figure \ref{5fig5}(c) has an equitable partition $\Pi''$ with cells $V_1=\{a,b\}, V_j=\{j-1\}, V_n=\{c,d\}$ for $2\leq j \leq n-1.$ The symmetrized quotient matrix relative to $\Pi''$ is equal to the adjacency matrix of the graph given in Figure \ref{5fig5}(b). Therefore, by \cite[Remark 5]{god8}, the result follows.
 \begin{cor}\label{5t31}
 Let $G$ be the graph obtained from $P_n$ by attaching two pendant vertices to each end vertex of $P_n.$ Then pretty good vertex state transfer occurs in $G$ from the vertices of 
 $P_n$ if and only if $n+1=2^k$ with $k\geq 1.$ 
 \end{cor}
For $n=2,$ the graph $G$ described in Corollary \ref{5t31} becomes a double star $S_{2,2},$ and the result is consistent with the result \cite[Theorem 5.3(b)]{fan}. Furthermore, when $n=3,$ the graph $G$ is the 1-sum of star graph $F_{2,2},$ and the conclusion agrees with \cite[Lemma 3.3]{hou}.

Note that a path $P_n$ with potential $1$ only at the end vertices can
 be realized as a symmetrized quotient graph of the even cycle. Hence,  using Theorem \ref{5th10} and \cite[Remark 5]{god8}, we have the following observation.

 \begin{cor}
     Let $P_n$ be a path on $n$ vertices with potential $1$ only at the end vertices. Then $P_n$ admits pretty good vertex state transfer if and only if $n=2^k,$ with $k\geq 1.$
 \end{cor}

 \section*{Disclosure statement}
  No potential conflict of interest was reported by the author(s).

\section*{Acknowledgements}
 We sincerely thank the reviewers for their insightful comments and valuable suggestions to improve the manuscript. 
S. Mohapatra is supported by the Department of Science and Technology (INSPIRE: IF210209).


\bibliographystyle{abbrv}
\bibliography{references}

\begin{thebibliography}{10}

\bibitem{alv}
R.~Alvir, S.~Dever, B.~Lovitz, J.~Myer, C.~Tamon, Y.~Xu, and H.~Zhan.
\newblock Perfect state transfer in {L}aplacian quantum walk.
\newblock {\em J. Algebraic Combin.}, 43(4):801--826, 2016.

\bibitem{bose}
S.~Bose.
\newblock Quantum communication through an unmodulated spin chain.
\newblock {\em Phys. Rev. Lett.}, 91:207901, 2003.

\bibitem{bro}
A.~E. Brouwer and W.~H. Haemers.
\newblock {\em Spectra of graphs}.
\newblock Universitext. Springer, New York, 2012.

\bibitem{chan2}
A.~Chan, G.~Coutinho, C.~Tamon, L.~Vinet, and H.~Zhan.
\newblock Quantum fractional revival on graphs.
\newblock {\em Discrete Appl. Math.}, 269:86--98, 2019.

\bibitem{cha3}
A.~Chan, W.~Drazen, O.~Eisenberg, M.~Kempton, and G.~Lippner.
\newblock Pretty good quantum fractional revival in paths and cycles.
\newblock {\em Algebr. Comb.}, 4(6):989--1004, 2021.

\bibitem{chan3}
A.~Chan, B.~Johnson, M.~Liu, M.~Schmidt, Z.~Yin, and H.~Zhan.
\newblock Laplacian fractional revival on graphs.
\newblock {\em Electron. J. Comb.}, 28(3):Paper 3.22, 2021.

\bibitem{che1}
Q.~Chen and C.~Godsil.
\newblock Pair state transfer.
\newblock {\em Quantum Inf. Process.}, 19(9):Paper No. 321, 30, 2020.

\bibitem{cou}
G.~Coutinho and C.~Godsil.
\newblock Perfect state transfer in products and covers of graphs.
\newblock {\em Linear Multilinear Algebra}, 64(2):235--246, 2016.

\bibitem{fan}
X.~Fan and C.~Godsil.
\newblock Pretty good state transfer on double stars.
\newblock {\em Linear Algebra Appl.}, 438(5):2346--2358, 2013.

\bibitem{farhi}
E.~Farhi and S.~Gutmann.
\newblock Quantum computation and decision trees.
\newblock {\em Phys. Rev. A (3)}, 58(2):915--928, 1998.

\bibitem{god1}
C.~Godsil.
\newblock State transfer on graphs.
\newblock {\em Discrete Math.}, 312(1):129--147, 2012.

\bibitem{god2}
C.~Godsil.
\newblock When can perfect state transfer occur?
\newblock {\em Electron. J. Linear Algebra}, 23:877--890, 2012.

\bibitem{god8}
C.~Godsil, S.~Kirkland, S.~Mohapatra, H.~Monterde, and H.~Pal.
\newblock Quantum walks on finite and bounded infinite graphs.
\newblock {\em arXiv preprint arXiv:2510.05306}, 2025.

\bibitem{god7}
C.~Godsil, S.~Kirkland, and H.~Monterde.
\newblock Perfect state transfer between real pure states.
\newblock {\em SIAM J. Matrix Anal. Appl.}, 46(3):2093--2115, 2025.

\bibitem{hou}
H.~Hou, R.~Gu, and M.~Tong.
\newblock Pretty good state transfer on 1-sum of star graphs.
\newblock {\em Open Math.}, 16(1):1483--1489, 2018.

\bibitem{jia}
M.~Jiang, X.~Liu, and J.~Wang.
\newblock Pair state transfer in tensor product and double cover.
\newblock {\em Discrete Appl. Math.}, 384:165--176, 2026.

\bibitem{kem1}
M.~Kempton, G.~Lippner, and S.-T. Yau.
\newblock Perfect state transfer on graphs with a potential.
\newblock {\em Quantum Inf. Comput.}, 17(3\&4):303--327, 2017.

\bibitem{kem2}
M.~Kempton, G.~Lippner, and S.-T. Yau.
\newblock Pretty good quantum state transfer in symmetric spin networks via
  magnetic field.
\newblock {\em Quantum Inf. Process.}, 16(9):210, 2017.

\bibitem{kim}
S.~Kim, H.~Monterde, B.~Ahmadi, A.~Chan, S.~Kirkland, and S.~Plosker.
\newblock A generalization of quantum pair state transfer.
\newblock {\em Quantum Inf. Process.}, 23(11):Paper No. 369, 2024.

\bibitem{kirk5}
S.~Kirkland and C.~M. van Bommel.
\newblock State transfer on paths with weighted loops.
\newblock {\em Quantum Inf. Process.}, 21(6):209, 2022.

\bibitem{pal11}
H.~Monterde and H.~Pal.
\newblock Perfect state transfer on graphs with clusters.
\newblock {\em arXiv e-prints}, page arXiv:2505.07982v3, 2025.

\bibitem{ojha1}
S.~Ojha and H.~Pal.
\newblock $ q $-laplacian state transfer on graphs with involutions.
\newblock {\em arXiv preprint arXiv:2509.20749}, 2025.

\bibitem{pal8}
H.~Pal.
\newblock Laplacian state transfer on graphs with an edge perturbation between
  twin vertices.
\newblock {\em Discrete Math.}, 345(7):112872, 2022.

\bibitem{pal4}
H.~Pal and B.~Bhattacharjya.
\newblock Pretty good state transfer on circulant graphs.
\newblock {\em Electron. J. Combin.}, 24(2):Paper No. 2.23, 13, 2017.

\bibitem{pal10}
H.~Pal and S.~Mohapatra.
\newblock Quantum pair state transfer on isomorphic branches.
\newblock {\em arXiv preprint arXiv:2402.07078}, 2024.

\bibitem{vin}
L.~Vinet and A.~Zhedanov.
\newblock Almost perfect state transfer in quantum spin chains.
\newblock {\em Phys. Rev. A}, 86(5):052319, 2012.

\end{thebibliography}

\end{document}